\documentclass[12pt,leqno]{article}
\usepackage{amsmath,amssymb,amscd,latexsym,paralist}
\usepackage[all]{xy}

\usepackage[all]{xy}

\newcommand{\G}{\mathbb{G}}

\newcommand{\oQ}{\overline{Q}}
\newcommand{\uA}{\underline{A}}
\newcommand{\Z}{{\mathbb{Z}}}

\newcommand{\bo}{\mathbf{1}}

\newcommand{\alg}{\mathrm{alg}}

\newcommand{\cont}{\mathrm{cont}}

\newcommand{\et}{\mathrm{\acute{e}t}}

\newcommand{\id}{\mathrm{id}}
\newcommand{\Maps}{\mathrm{Maps}\,}
\renewcommand{\mod}{\;\mathrm{mod}\;}

\newcommand{\spec}{\mathrm{spec}\,}
\newcommand{\Coker}{\mathrm{Coker}\,}
\newcommand{\Hom}{\mathrm{Hom}}

\newcommand{\Ker}{\mathrm{Ker}\,}

\newcommand{\tJ}{\tilde{J}}

\newcommand{\ssp}{\mathrm{sp}\,}

\newcommand{\tors}{\mathrm{tors}}
\newcommand{\Ah}{{\mathcal A}}

\newcommand{\Gh}{{\mathcal G}}
\newcommand{\Oh}{{\mathcal O}}
\newcommand{\Sh}{{\mathcal S}}

\newcommand{\eo}{\mathfrak{o}}
\newcommand{\iso}{\stackrel{\sim}{\longrightarrow}}
\newcommand{\tei}{\, | \,}

\newtheorem{theorem}{Theorem}
\newtheorem{coronatheorem}[theorem]{$P$-adic Corona theorem}
\newtheorem{lemma}[theorem]{Lemma}
\newtheorem{prop}[theorem]{Proposition}

\newtheorem{claim}[theorem]{Claim}
\newtheorem{condition}[theorem]{Condition}
\newtheorem{conj}[theorem]{Conjecture}

\newenvironment{proof}{\noindent {\bf Proof}}{\mbox{}\hfill$\Box$}
\begin{document}
\title{Invariant functions on $p$-divisible groups and the $p$-adic Corona problem}
\author{Christopher Deninger}
\date{\ }
\maketitle
\thispagestyle{empty}

\section{Introduction}
\label{sec:1}

In this note we are concerned with $p$-divisible groups $G = (G_{\nu})$ over a complete discrete valuation ring $R$. We assume that the fraction field $K$ of $R$ has characteristic zero and that the residue field $k = R /\pi R$ is perfect of positive characteristic $p$. 

Let $C$ be the completion of an algebraic closure of $K$ and denote by $\eo = \eo_C$ its ring of integers. The group $G_{\nu} (\eo)$ acts on $G_{\nu} \otimes \eo$ by translation. Since $G_{\nu} \otimes K$ is \'etale the $G_{\nu} (C)$-invariant functions on $G_{\nu} \otimes C$ are just the constants. Using the counit it follows that the natural inclusion
\[
\eo \iso \Gamma (G_{\nu} \otimes \eo , \Oh)^{G_{\nu} (\eo)}
\]
is an isomorphism. We are interested in an approximate $\mod \pi^n$-version of this statement. Set $\eo_n = \eo / \pi^n \eo$ for $n \ge 1$. The group $G_{\nu} (\eo)$ acts by translation on $G_{\nu} \otimes \eo_n$ for all $n$.

\begin{theorem}
  \label{t1}
Assume that the dual $p$-divisible group $G'$ is at most one-dimensional and that the connected-\'etale exact sequence for $G'$ splits over $\eo$. Then there is an integer $t \ge 1$ such that the cokernel of the natural inclusion
\[
\eo_n \hookrightarrow \Gamma (G_{\nu} \otimes \eo_n , \Oh)^{G_{\nu} (\eo)}
\]
is annihilated by $p^t$ for all $\nu$ and $n$. 
\end{theorem}

The example of $\G_m = (\mu_{p^{\nu}})$ in section \ref{sec:2} may be helpful to get a feeling for the statement. 

We expect the theorem to hold without any restriction on the dimension of $G$ as will be explained later. Its assertion is somewhat technical but the proof may be of interest because it combines some of the main results of Tate on $p$-divisible groups with van~der~Put's solution of his one-dimensional $p$-adic Corona problem.

The classic corona problem concerns the Banach algebra $H^{\infty} (D)$ of bounded analytic functions on the open unit disc $D$. The points of $D$ give maximal ideals in $H^{\infty} (D)$ and hence points of the Gelfand spectrum $\hat{D} = \ssp H^{\infty} (D)$. The question was whether $D$ was dense in $\hat{D}$, (the set $\hat{D} \setminus \overline{D}$ being the ``corona''). This was settled affirmatively by Carleson \cite{C}. The analogous question for the polydisc $D^d$ is still open for $d \ge 2$. An equivalent condition for $D^d$ to be dense in $\ssp H^{\infty} (D^d)$ is the following one, \cite{H}, Ch. 10:

\begin{condition}
  \label{t2}
If $f_1 , \ldots , f_n$ are bounded analytic functions in $D^d$ such that for some $\delta > 0$ we have
\[
\max_{1 \le i \le n} |f_i (z)| \ge \delta \quad \mbox{for all} \; z \in D^d \; ,
\]
then $f_1 , \ldots , f_n$ generate the unit ideal of $H^{\infty} (D^d)$.
\end{condition}

In \cite{P} van~der~Put considered the analogue of condition \ref{t2} with $H^{\infty} (D^d)$ replaced by the algebra of bounded analytic $C$-valued functions on the $p$-adic open polydisc $\Delta^d$ in $C^d$, i.e. by the algebra
\[
C \langle X_1 , \ldots , X_d \rangle = \eo [[ X_1 , \ldots , X_d]] \otimes_{\eo} C \; .
\]
He called this $p$-adic version of condition \ref{t2} the $p$-adic Corona problem and verified it for $d = 1$. The general case $d \ge 1$ was later treated by Bartenwerfer \cite{B} using his earlier results on rigid cohomology with bounds.

In the proof of theorem \ref{t1} applying Tate's results from \cite{T} we are led to a question about certain ideals in $C \langle X_1 , \ldots , X_d \rangle$, which for $d = 1$ can be reduced to van~der~Put's $p$-adic Corona problem. For $d \ge 2$, I did not succeed in such a reduction. However it seems possible that a generalization of Bartenwerfer's theory might settle that question. 

It should be mentioned that van~der~Put's term ``$p$-adic Corona problem'' for the $p$-adic analogue of condition \ref{t2} is somewhat misleading. Namely as pointed out in \cite{EM} a more natural analogue whould be the question whether $\Delta^d$ was dense in the Berkovich space of $C \langle X_1 , \ldots , X_d \rangle$. This is not known, even for $d = 1$. The difference between the classic and the $p$-adic cases comes from the fact discovered by van~der~Put that contrary to $H^{\infty} (D^d)$ the algebra $C \langle X_1 , \ldots, X_d \rangle$ contains maximal ideals of infinite codimension.

While working on this note I had helpful conversations and exchanges with Siegfried Bosch, Alain Escassut, Peter Schneider, Annette Werner and Thomas Zink. I would like to thank them all very much.

\section{An example and other versions of the theorem}
\label{sec:2}

Consider an affine group scheme $\Gh$ over a ring $S$ with Hopf-algebra $\Ah = \Gamma (\Gh , \Oh)$, comultiplication $\mu : \Ah \to \Ah \otimes_S \Ah$ and counit $\varepsilon : \Ah \to S$. The operation of $\Gh (S) = \Hom_S (\Ah , S)$ on $\Gamma (\Gh , \Oh)$ by translation is given by the map
\begin{equation}
  \label{eq:1}
  \Gh (S) \times \Ah \to \Ah \; , \; (\chi , a) \mapsto (\chi \otimes \id) \mu (a)
\end{equation}
where $(\chi \otimes \id) \mu$ is the composition
\[
\Ah \xrightarrow{\mu} \Ah \otimes_S \Ah \xrightarrow{\chi \otimes \id} S \otimes_S \Ah = \Ah \; .
\]
Given a homomorphism of groups $P \to \Gh (S)$ we may view $\Ah$ as a $P$-module. The composition $S \to \Ah \xrightarrow{\varepsilon} S$ being the identity we have an isomorphism
\begin{equation}
  \label{eq:2}
  \Ker (\Ah^P \xrightarrow{\varepsilon} S) \iso \Ah^P / S \quad \mbox{mapping $a$ to} \; a + S \; .
\end{equation}
The inverse sends $a + S$ to $a - \varepsilon (a) \cdot 1$.

{\bf Example} The theorem is true for $\G_m = (\mu_{p^{\nu}})$.

\begin{proof}
  Set $V = \eo_n [X , X^{-1}] / (X^{p^{\nu}} - 1)$. Applying formulas \eqref{eq:1} and \eqref{eq:2} with $\Gh = \mu_{p^{\nu}} \otimes \eo_n$ and $P = \mu_{p^{\nu}} (\eo) \to \Gh (\eo_n)$ we see that the cokernel of the map
  \begin{equation}
    \label{eq:3}
    \eo_n \to \Gamma (\mu_{p^{\nu}} \otimes \eo_n , \Oh)^{\mu_{p^{\nu}} (\eo)}
  \end{equation}
is isomorphic to the $\eo_n$-module:
\[
\{ \oQ \in V \tei \oQ (\zeta X) = \oQ (X) \; \mbox{for all} \; \zeta \in \mu_{p^{\nu}} (\eo) \; \mbox{and} \; \oQ (1) = 0 \} \; .
\]
Lift $\oQ$ to a Laurent polynomial $Q = \sum_{\mu \in \Sh} a_{\mu} X^{\mu}$ in $\eo [X , X^{-1}]$ where $\Sh = \{ 0 , \ldots , p^{\nu}-1 \}$.

Then we have:
\begin{equation}
  \label{eq:4}
  (\zeta^{\mu} - 1) a_{\mu} \equiv 0 \mod \pi^n \quad \mbox{for} \; \mu \in \Sh \; \mbox{and} \; \zeta \in \mu_{p^{\nu}} (\eo)
\end{equation}
and
\begin{equation}
  \label{eq:5}
  \sum_{\mu \in \Sh} a_{\mu} \equiv 0 \mod \pi^n  \; .
\end{equation}
For any non-zero $\mu$ in $\Sh$ choose $\zeta \in \mu_{p^{\nu}} (\eo)$ such that $\zeta^{\mu} \neq 1$. Then $\zeta^{\mu} -1$ divides $p$ in $\eo$ and hence \eqref{eq:4} implies that $p a_{\mu} \equiv 0 \mod \pi^n$ for all $\mu \neq 0$. Using \eqref{eq:5} it follows that we have $p a_0 \equiv 0 \mod \pi^n$ as well. Hence $pQ \mod \pi^n$ is zero and therefore $p \oQ = 0$ as well. Thus $p$ annihilates the $\eo_n$-module \eqref{eq:3} for all $\nu \ge 1$ and $n \ge 1$. 
\end{proof}

Now assume that $S = R$ and that $\Gh / R$ is a finite, flat group scheme. Consider the Cartier dual $\Gh' = \spec \Ah'$ where $\Ah' = \Hom_R (\Ah , R)$. The perfect pairing of finite free $\eo_n$-modules
\begin{equation}
  \label{eq:6}
  (\Ah \otimes \eo_n) \times (\Ah' \otimes \eo_n) \to \eo_n
\end{equation}
induces an isomorphism
\begin{equation}
  \label{eq:7}
  \Ker ((\Ah \otimes \eo_n)^{\Gh (\eo)} \xrightarrow{\varepsilon} \eo_n) \iso \Hom_{\eo_n} ((\Ah' \otimes \eo_n)_{\Gh (\eo)} / \eo_n , \eo_n) \; .
\end{equation}
Using \eqref{eq:2} it follows that if $p^t$ annihilates $(\Ah' \otimes \eo_n)_{\Gh (\eo)} / \eo_n$ then $p^t$ annihilates $(\Ah \otimes \eo_n)^{\Gh (\eo)} / \eo_n$ as well. (The converse is not true in general.)

Hence theorem \ref{t1} follows from the next result (applied to the dual $p$-divisible group).

\begin{theorem}
  \label{t3}
Assume that the $p$-divisible group $G$ is at most one-dimensional and that the connected-\'etale exact sequence for $G$ splits over $\eo$. Then there is an integer $t \ge 1$ such that $p^t$ annihilates the cokernel of the natural map
\[
\eo_n \to \Gamma (G_{\nu} \otimes \eo_n , \Oh)_{G'_{\nu} (\eo)}
\]
for all $\nu$ and $n$.
\end{theorem}

For a finite flat group scheme $\Gh = \spec \Ah$ over a ring $S$, the group
\[
\Gh' (S) = \Hom_{S-\alg} (\Hom_S (\Ah , S),S) \subset \Ah
\]
consists of the group-like elements in $\Ah$ i.e. the units $a$ in $\Ah$ with $\mu (a) = a \otimes a$. In this way $\Gh' (S)$ becomes a subgroup of the unit group $\Ah^*$ and hence $\Gh' (S)$ acts on $\Ah$ by multiplication. On the other hand $\Gh' (S)$ acts on $\Gh'$ by translation, hence on $\Ah' = \Gamma (\Gh' , \Oh)$ and hence on $\Ah'' = \Ah$. Using \eqref{eq:1} one checks that the two actions of $\Gh' (S)$ on $\Ah$ are the same. This leads to the following description of the cofixed module in theorem \ref{t3}. Set $A_{\nu} = \Gamma (G_{\nu} , \Oh)$ and let $J_{\nu}$ be the ideal in $A_{\nu} \otimes_R \eo$ generated  by the elements $h-1$ with $h$ group-like in this Hopf-algebra over $\eo$. Thus $J_{\nu}$ is also the $\eo$-submodule of $A_{\nu} \otimes_R \eo$ generated by the elements $ha - a$ for $h \in G'_{\nu} (\eo)$ and $a \in A_{\nu} \otimes_R \eo$. Then we have the formula
\begin{equation}
  \label{eq:8}
  \Gamma (G_{\nu} \otimes \eo_n , \Oh)_{G'_{\nu} (\eo)} = (A_{\nu} \otimes_R \eo_n) / J_{\nu} (A_{\nu} \otimes_R \eo_n) \; .
\end{equation}
This implies an isomorphism:
\begin{equation}
  \label{eq:9}
  \Coker (\eo_n \to \Gamma (G_{\nu} \otimes \eo_n, \Oh)_{G'_{\nu} (\eo)}) = \Coker (\eo \to (A_{\nu} \otimes_R \eo) / J_{\nu}) \otimes_{\eo} \eo_n \; .
\end{equation}
Hence theorem \ref{t3} and therefore also theorem \ref{t1} follow from the next claim:

\begin{claim}
  \label{t4}
For a $p$-divisible group $G = (G_{\nu})$ as in theorem \ref{t3} there exists an integer $t \ge 1$ such that $p^t$ annihilates the cokernel of the natural map \\
$\eo \to (A_{\nu} \otimes_R \eo) / J_{\nu}$ for all $\nu \ge 1$.
\end{claim}

As a first step in the proof of claim \ref{t4} we reduce to the case where $G$ is either \'etale or connected. For simplicity set $\Gh = G_{\nu} \otimes_R \eo = \spec \Ah$ and define $\Gh^0 , \Gh^{\et} , \Ah^0 , \Ah^{\et}$ similarly. By assumption we have isomorphisms $\Gh = \Gh^0 \times_{\eo} \Gh^{\et}$ and $\Ah = \Ah^0 \otimes_{\eo} \Ah^{\et}$ as group schemes, resp. Hopf-algebras over $\eo$. There is a compatible splitting of the group-like elements over $\eo$:
\[
\Gh' (\eo) = \Gh^{0'} (\eo) \times \Gh^{\et '} (\eo) \; .
\]
For elements
\[
h^0 \in \Gh^{0'} (\eo) \subset \Ah^0 \quad \mbox{and} \quad h^{\et} \in \Gh^{\et '} (\eo) \subset \Ah^{\et}
\]
consider the identity:
\[
h^0 \otimes h^{\et} - 1 = h^0 \otimes (h^{\et} - 1) + (h^0 - 1) \otimes 1 \quad \mbox{in} \; \Ah \; .
\]
It implies that we have
\[
J = \Ah^0 \otimes J^{\et} + J^0 \otimes \Ah^{\et} \quad \mbox{in} \; \Ah
\]
where $J$ is the ideal of $\Ah$ generated by the elements $h-1$ for $h \in \Gh' (\eo)$ and $J^0 , J^{\et}$ are defined similarly. Hence we have natural surjections
\[
\Ah^0 / J^0 \otimes \Ah^{\et} / J^{\et} \to \Ah / J
\]
and
\[
\Coker (\eo \to \Ah^0 / J^0) \otimes \Coker (\eo \to \Ah^{\et} / J^{\et}) \to \Coker (\eo \to \Ah / J) \; .
\]
Hence it suffices to prove claim \ref{t4} in the cases where $G$ is either connected or \'etale. The \'etale case is straightforeward: We have $G \otimes_R \eo = ((\underline{\Z / p^{\nu}})^h)_{\nu \ge 0}$ where for an abstract group $A$ we denote by $\uA$ the corresponding \'etale group scheme. Hence $G'_{\nu} = \mu^h_{p^{\nu}}$ and $G'_{\nu} (\eo) = \mu_{p^{\nu}} (\eo)^h$. The inclusion 
\[
\mu_{p^{\nu}} (\eo)^h = \Hom ((\Z / p^{\nu})^h , \eo^*) \subset \Maps ((\Z / p^{\nu})^h , \eo) = A_{\nu} \otimes \eo
\]
identifies $\mu_{p^{\nu}} (\eo)^h$ with the group like elements in $A_{\nu} \otimes \eo$. 

The ideal $J_{\nu}$ of $A_{\nu} \otimes \eo$ is given by:
\[
J_{\nu} = (\chi_{\zeta} - 1 \tei \zeta \in \mu_{p^{\nu}} (\eo)^h)
\]
where $\chi_{\zeta}$ is the character of $(\Z / p^{\nu})^h$ defined by the equation
\[
\chi_{\zeta} ((a_1 , \ldots , a_h)) = \zeta^{a_1}_1 \cdots \zeta^{a_h}_h \quad \mbox{where} \; \zeta = (\zeta_1 , \ldots , \zeta_h) \; .
\]
The functions $\delta_a$ for $a \in (\Z / p^{\nu})^h$ given by $\delta_a (a) = 1$ and $\delta_a (b) = 0$ if $b \neq a$ generate $A_{\nu} \otimes \eo$ as an $\eo$-module. For $a \neq 0$ choose $\zeta \in \mu_{p^{\nu}} (\eo)^h$ with $\zeta^a \neq 1$. Then we have $p = (\zeta^a - 1) \beta$ for some $\beta \in \eo$. Define $f_a \in A_{\nu} \otimes \eo$ by setting
\[
f_a (a) = \beta \quad \mbox{and} \quad f_a (b) = 0 \; \mbox{for} \; b \neq a \; .
\]
We then find:
\[
f_a (\chi_{\zeta} - 1) = p \delta_a \quad \mbox{in} \; A_{\nu} \otimes \eo \; .
\]
Hence we have $p \delta_a \in J_{\nu}$ for all $a \neq 0$ and therefore $p$ annihilates $\Coker (\eo \to (A_{\nu} \otimes \eo) / J_{\nu})$.

The next two sections are devoted to the much more interesting case where $G$ is connected.

\section{The connected case I ($p$-adic Hodge theory)}
\label{sec:3}
In this section we reduce the assertion of claim \ref{t4} for connected $p$-divisible groups of arbitrary dimension to an assertion on ideals in $C \langle X_1 , \ldots , X_d \rangle$. For this reduction we use theorems of Tate in \cite{T}.

Thus let $G = (G_{\nu})$ be a connected $p$-divisible group of dimension $d$ over $R$ and set $A = \varprojlim_{\nu} A_{\nu}$ where $G_{\nu} = \spec A_{\nu}$. 

Consider the projective limit $A = \varprojlim A_n$ with the topology inherited from the product topology $\prod A_n$ where the $A_n$'s are given the $\pi$-adic topology. This topology on $A$ is the one defined by the $R$-submodules $K_n + \pi^k A$ for $n , k \ge 1$ where $K_n = \Ker (A \to A_n)$. Equivalently it is defined by the spaces $K_n + \pi^n A$ for $n \ge 1$. In \cite{T} section (2.2) it is shown that $A$ is isomorphic to $R [[X_1 , \ldots , X_d]]$ as a topological $R$-algebra. If $M$ denotes the maximal ideal of $A$, then according to \cite{T} Lemma 0 the topology of $A$ coincides with the $M$-adic topology. Let $A \hat{\otimes}_R \eo$ be the completion of $A \otimes_R \eo$ with respect to the linear topology on $A \otimes_R \eo$ given by the subspaces $M^n \otimes_R \eo + A \otimes_R \pi^n \eo$. 

\begin{lemma}
  \label{t5n}
We have
\[
\varprojlim (A_n \otimes_R \eo) = A \hat{\otimes}_R \eo = \eo [[ X_1 , \ldots , X_d ]]
\]
as topological rings.
\end{lemma}

\begin{proof}
Consider the isomorphisms
  \begin{eqnarray*}
    \varprojlim_n (A_n \otimes_R \eo) & = & \varprojlim_n (A_n \otimes_R (\varprojlim_k \eo / \pi^k \eo)) \\
& \overset{(1)}{=} & \varprojlim_n \varprojlim_k (A_n \otimes_R \eo / \pi^k \eo) \\
& = & \varprojlim_n \varprojlim_k (A \otimes_R \eo) / ((K_n + \pi^k A) \otimes_R \eo + A \otimes_R \pi^k \eo) \\
& \overset{(2)}{=} & \varprojlim_n (A \otimes_R \eo) / ((K_n + \pi^n A) \otimes_R \eo + A \otimes_R \pi^n \eo) \\
& \overset{(3)}{=} & \varprojlim_n (A \otimes_R \eo) / (M^n \otimes_R \eo + A \otimes_R \pi^n \eo) \\
& = & A \hat{\otimes}_R \eo \\
& \overset{(4)}{=} & \eo [[ X_1 , \ldots , X_d]] \; .
  \end{eqnarray*}
Here (1) holds because $\varprojlim$ commutes with finite direct sums, (2) is true by cofinality, (3) holds because the topology on $A$ can also be described as the $M$-adic topology. Finally (4) follows from the definition of $A \hat{\otimes}_R \eo$ and the fact that $A = R [[X_1 , \ldots , X_d]]$.
\end{proof}

The $\eo$-algebra $A \hat{\otimes}_R \eo = \varprojlim_{\nu} (A_{\nu} \otimes_R \eo)$ contains the ideal $\tJ = \varprojlim_{\nu} J_{\nu}$. 

\begin{claim}
  \label{t5}
We have
\[
A \hat{\otimes}_R \eo / (\eo + \tJ) = \varprojlim_{\nu} (A_{\nu} \otimes_R \eo / (\eo + J_{\nu})) \; .
\]
\end{claim}

\begin{proof}
  The inclusion $G_{\nu} \subset G_{\nu+1}$ corresponds to a surjection of Hopf-algebras $A_{\nu+1} \to A_{\nu}$. Hence $A_{\nu+1} \otimes_R \eo \to A_{\nu} \otimes_R \eo$ is surjective as well and group-like elements are mapped to group-like elements. The map on group-like elements is surjective because it corresponds to the surjective map $G'_{\nu+1} (\eo) \to G'_{\nu} (\eo)$. Note here that $G'_{\mu} (\eo) = G'_{\mu} (C)$ for all $\mu$. It follows that the map $J_{\nu+1} \to J_{\nu}$ is surjective as well. In the exact sequence of projective systems
\[
0 \to (\eo + J_{\nu}) \to (A_{\nu} \otimes_R \eo) \to (A_{\nu} \otimes_R \eo / (\eo + J_{\nu})) \to 0
\]
the system $(\eo + J_{\nu})$ is therefore Mittag--Leffler. Hence the sequence of projective limits is exact and the claim follows because the sum $\eo + J_{\nu}$ is direct: Group-like elements of $A_{\nu} \otimes_R \eo$ are mapped to $1$ by the counit $\varepsilon_{\nu}$. Therefore we have
\begin{equation}
  \label{eq:10}
  J_{\nu} \subset I_{\nu} := \Ker (\varepsilon_{\nu} : A_{\nu} \otimes_R \eo \to \eo) \; .
\end{equation}
The sum $\eo + I_{\nu}$ being direct we are done.
\end{proof}

Because of claim \ref{t5} and the surjectivity of the maps $A_{\nu+1} \otimes_R \eo \to A_{\nu} \otimes_R \eo$, claim \ref{t4} for connected groups is equivalent to the next assertion:

\begin{claim}
  \label{t6}
Let $G$ be a connected $p$-divisible group with $\dim G \le 1$. Then there is some $t \ge 1$ such that $p^t$ annihilates 
\[
\Coker (\eo \to A \hat{\otimes}_R \eo / \tJ) \; .
\]
\end{claim}

For connected $G$ of arbitrary dimension consider the Tate module of $G'$
\[
TG' = \varprojlim_{\nu} G'_{\nu} (C) = \varprojlim_{\nu} G'_{\nu} (\eo) \subset \varprojlim_{\nu} (A_{\nu} \otimes_R \eo) = A \hat{\otimes}_R \eo \; .
\]
Let $J$ be the ideal of $A \hat{\otimes}_R \eo $ generated by the elements $h-1$ for $h \in TG'$. The image of $J$ under the reduction map $A \hat{\otimes}_R \eo \to A_{\nu} \otimes_R \eo$ lies in $J_{\nu}$. It follows that $J \subset \tJ$. With $I_{\nu}$ as in \eqref{eq:10} we set $I = \varprojlim_{\nu} I_{\nu}$, an ideal in $A \hat{\otimes}_R \eo$. We have $J \subset \tJ \subset I$ because of \eqref{eq:10}. Since $A \hat{\otimes}_R \eo = \eo \oplus I$, we get a surjection
\begin{equation}
  \label{eq:11}
  I / J \twoheadrightarrow \Coker (\eo \to A\hat{\otimes}_R \eo / \tJ) \; .
\end{equation}
Thus claim \ref{t6} will be proved if we can show that $p^t I \subset J$ at least for $\dim G = 1$. The construction in \cite{T} section (2.2) shows that under the isomorphism of $\eo$-algebras
\[
A \hat{\otimes}_R \eo = \eo [[X_1 , \ldots , X_d]] \quad \mbox{we have} \; I = (X_1 , \ldots , X_d) \; .
\]
We will view the elements of $A \hat{\otimes}_R \eo$ and in particular those of $J$ as analytic functions on the open $d$-dimensional polydisc 
\[
\Delta^d = \{ x \in C^d \tei |x_i| < 1 \; \mbox{for all} \; i \} \; .
\]
Because of the inclusion $J \subset I$ all functions in $J$ vanish at $0 \in \Delta^d$. There are no other common zeroes:

\begin{prop}[Tate]
  \label{t7}
The zero set of $J$ in $\Delta^d$ consists only of the origin $o \in \Delta^d$.
\end{prop}

\begin{proof}
  The $\eo$-valued points of the $p$-divisible group $G$,
\[
G (\eo) = \varprojlim_i \varinjlim_{\nu} G_{\nu} (\eo / \pi^i \eo)
\]
can be identified with continuous $\eo$-algebra homomorphisms
\[
G (\eo) = \Hom_{\cont , \alg} (A , \eo) = \Hom_{\cont , \alg} (A \hat{\otimes}_R \eo , \eo) \; .
\]
Moreover we have a homeomorphism 
\[
\Delta^d \iso G (\eo) \quad \mbox{via} \quad x \mapsto (f \mapsto f (x)) \; .
\]
Here $f \in A\hat{\otimes}_R \eo$ is viewed as a formal power series over $\eo$. The group structure on $G (\eo)$ induces a Lie group structure on $\Delta^d$ with $0 \in \Delta^d$ corresponding to $1 \in G (\eo)$. Let $U$ be the group of $1$-units in $\eo$. Proposition 11 of \cite{T} asserts that the homomorphism of Lie groups
\begin{equation}
\label{eq:12}
\alpha : \Delta^d = G (\eo) \to \Hom_{\cont} (TG' , U) \; , \; x \mapsto (h \mapsto h (x))
\end{equation}
is {\it injective}. Note here that $TG' \subset A \hat{\otimes}_R \eo$. Let $x \in \Delta^d$ be a point in the zero set of $J$. Then we have $(h-1) (x) = 0$ i.e. $h (x) = 1$ for all $h \in TG'$. Hence $x$ maps to $1 \in \Hom_{\cont} (TG' , U)$. Since $\alpha$ is injective, it follows that we have $x = 0$.
\end{proof}

If a Hilbert Nullstellensatz were true in $C \langle X_1 , \ldots , X_d \rangle$ we could conclude that we had $\sqrt{J \otimes C} = I \otimes C$ and with further arguments from \cite{T} we would get $p^t I \subset J$. However the Nullstellensatz does not hold in the ring $C \langle X_1 , \ldots , X_d \rangle$. 

In the next section we will provide a replacement which is proved for $d = 1$ and conjectured for $d \ge 2$. In order to apply it to the ideal $J \otimes C$ in $C \langle X_1 , \ldots , X_d \rangle$ we need to know the following assertion which is stronger than proposition \ref{t7}. For $x \in C^m$ set $\| x \| = \max_i |x_i|$. 

\begin{prop}
  \label{t8}
Let $h_1 , \ldots , h_r$ be a $\Z_p$-basis of $TG' \subset \eo [[ X_1 , \ldots , X_d ]]$ and set $H (x) = (h_1 (x) , \ldots , h_r (x))$ and $\bo = (1 , \ldots , 1)$. Then there is a constant $\delta > 0$ such that we have:
\[
\| H (x) - \bo \| \ge \delta \|x \| \quad \mbox{for all} \; x \in \Delta^d \; .
\]
\end{prop}

\begin{proof}
  The $\Z_p$-rank $r$ of $TG'$ is the height of $G'$ and hence we have $r \ge d = \dim G$. Consider the following diagram $(*)$ on p. 177 of \cite{T}:
\[
\xymatrix{
1 \ar[r] & G (\eo)_{\tors} \ar[r] \ar[d]^{\wr}_{\alpha_0} & G (\eo) \ar[r]^L \ar@{^{(}->}[d]_{\alpha} & t_G (C) \ar[r] \ar@{^{(}->}[d]_{d \alpha} & 0 \\
1 \ar[r] & \Hom (TG' , U_{\tors}) \ar[r] & \Hom (TG' , U) \ar[r]^{\log_*} & \Hom (TG' , C) \ar[r] & 0 .
}
\]
Here the $\Hom$-groups refer to continuous homomorphisms and the map $\alpha$ was defined in equation \eqref{eq:12} above. The map $L$ is the logarithm map to the tangent space $t_G (C)$ of $G$ and $\log_*$ is induced by $\log : U \to C$. According to \cite{T} proposition 11 the maps $\alpha$ and $d\alpha$ are injective and $\alpha_0$ is bijective. It will suffice to prove the following two statements:
\begin{enumerate}
\item [I)] For any $\varepsilon > 0$ there is a constant $\delta (\varepsilon) > 0$ such that
\[
\| H (x) - \bo \| \ge \delta (\varepsilon) \; \mbox{for all} \; x \in \Delta^d \; \mbox{with} \; \| x \| \ge \varepsilon \; .
\]
\item[II)] There are $\varepsilon > 0$ and $a > 0$ such that
\[
\| H (x) - \bo \| \ge a \|x \| \; \mbox{for all} \; x \in \Delta^d \; \mbox{with} \; \| x \| \le \varepsilon \; .
\]
\end{enumerate}
Identifying $G (\eo)$ with $\Delta^d$ where we write the induced group structure on $\Delta^d$ as $\oplus$, and identifying  $TG'$ with $\Z^r_p$ via the choice of the basis $h_1 , \ldots , h_r$, the above diagram becomes the following one where $A = dH$ and $H_0$ is the restriction of $H$ to $(\Delta^d)_{\tors}$
\[
\xymatrix{
0 \ar[r] &  (\Delta^d)_{\tors} \ar[r] \ar[d]^{\wr}_{H_0} & \Delta^d \ar[r]^L \ar@{^{(}->}[d]_H & C^d \ar[r] \ar@{^{(}->}[d]^{A = dH} & 0 \\
1 \ar[r] & U^r_{\tors} \ar[r] & U^r \ar[r]^{\log} & C^r \ar[r] & 0 .
}
\]
Assume that assertion $I$ is wrong for some $\varepsilon > 0$. Then there is a sequence $x^{(i)}$ of points in $\Delta^d$ with $\| x^{(i)} \| \ge \varepsilon$ such that $H (x^{(i)}) \to \bo$ for $i \to \infty$. It follows that $A (L (x^{(i)}) = \log H (x^{(i)}) \to 0$ for $i \to \infty$. Since $A$ is an injective linear map between finite dimensional $C$-vector spaces, there exists a constant $a > 0$ such that we have
\begin{equation}
  \label{eq:13}
  \| A (v) \| \ge a \|v\| \quad \mbox{for all} \; v \in C^d \; .
\end{equation}
Hence we see that $L (x^{(i)}) \to 0$ for $i \to \infty$. Since $L$ is a local homeomorphism, there exists a sequence $y^{(i)} \to 0$ in $\Delta^d$ with $L (x^{(i)}) = L (y^{(i)})$ for all $i$. The sequence $z^{(i)} = x^{(i)} \ominus y^{(i)}$ in $\Delta^d$ satisfies $L (z^{(i)}) = 0$ and hence lies in $(\Delta^d)_{\tors}$. We have $H_0 (z^{(i)}) = H (x^{(i)}) H (y^{(i)})^{-1}$. Moreover $H (x^{(i)}) \to \bo$ by assumption and $H (y^{(i)}) \to \bo$ since $y^{(i)} \to 0$. Hence $H_0 (z^{(i)}) \to 1$ and therefore $H_0 (z^{(i)}) = 1$ for all $i \gg 0$ since the subspace topology on $U_{\tors} \subset U$ is the discrete topology. The map $H_0$ being bijective we find that $z^{(i)} = 0$ for $i \gg 0$ and therefore $x^{(i)} = y^{(i)}$ for $i \gg 0$. This implies that $x^{(i)} \to 0$ for $i \to \infty$ contradicting the assumption $\| x^{(i)} \| \ge \varepsilon$ for all $i$. Hence assertion I) is proved.

We now turn to assertion II). Set $X = (X_1 , \ldots , X_d)$. Then we have
\[
H(X) = \bo + AX + (\deg \ge 2) \; .
\]
Componentwise this gives for $1 \le j \le r$
\[
h_j (x) - 1 = \sum^d_{i=1} a_{ij} x_i + (\deg \ge 2)_j \; .
\]
Let $a$ be the constant from equation \eqref{eq:13} and  choose $\varepsilon > 0$, such that for $\| x \| \le \varepsilon$ we have
\[
\| (\deg \ge 2)_j \| \le \frac{a}{2} \| x \| \quad \mbox{for} \; 1 \le j \le r \; .
\]
For any $x$ with $\| x \| < \varepsilon$, according to \eqref{eq:13} there is an index $j$ with
\[
\Big| \sum^d_{i=1} a_{ij} x_i \Big| \ge a \|x \| \; .
\]
This implies that we have
\[
| h_j (x) - 1 | = \Big| \sum^d_{i=1} a_{ij} x_i + (\deg \ge 2)_j \Big| = \Big| \sum^d_{i=1} a_{ij} x_i \Big| \ge a \| x \| 
\]
and hence
\[
\| H (x) - \bo \| \ge a \| x \| \; .
\]
\end{proof}

\section{The connected case II (the $p$-adic Corona problem)}
\label{sec:4}

As remarked in the previous section we need a version of the Hilbert Nullstellensatz in $C \langle X_1 , \ldots , X_d \rangle$ for the case where the zero set is $\{ 0 \} \subset \Delta^d$. The only result for $C \langle X_1 , \ldots , X_d \rangle$ in the spirit of the Nullstellensatz that I am aware of concerns an empty zero set:

\begin{coronatheorem}[van~der~Put, Bartenwerfer]
  \label{t9}
For $f_1 , \ldots , f_n$ in $C \langle X_1 , \ldots , X_d \rangle$ the following conditions are equivalent:
\begin{enumerate}
\item [1)] The functions $f_1 , \ldots , f_n$ generate the $C$-algebra $C \langle X_1 , \ldots , X_d \rangle$.
\item[2)] There is a constant $\delta > 0$ such that
\[
\max_{1 \le j \le n} |f_j (x)| \ge \delta \quad \mbox{for all} \; x \in \Delta^d \; .
\]
\end{enumerate}
\end{coronatheorem}

It is clear that the first condition implies the second. The non-trivial implication was proved by van~der~Put for $d = 1$ in \cite{P} and by Bartenwerfer in general, c.f. \cite{B}. Both authors give a more precise statement of the theorem where the norms of possible functions $g_j$ with $\sum_j f_j g_j = 1$ are estimated.

Consider the following conjecture which deals with the case where the zero set may contain $\{ 0 \}$.

\begin{conj}
  \label{t10}
For $g_1 , \ldots , g_n$ in $C \langle X_1 , \ldots , X_d \rangle$ the following conditions are equivalent:
\begin{enumerate}
\item [1)] $(g_1 , \ldots , g_n) \supset (X_1 , \ldots , X_d)$.
\item[2)] There is a constant $\delta > 0$ such that
  \begin{equation}
    \label{eq:14}
    \max_{1 \le j \le n} |g_j (x)| \ge \delta \| x \| \quad \mbox{for all} \; x \in \Delta^d \; .
  \end{equation}
\end{enumerate}
\end{conj}

As above, immediate estimates show that the first condition implies the second. Note also that if some $g_j$ does not vanish at $x = 0$ we have 
\[
\max_{1 \le j \le n} |g_j (x)| \ge \delta' > 0 \; \mbox{in a neighborhood of} \; x = 0 \; . 
\]
Together with \eqref{eq:14} this implies that 
\[
\max_{1 \le j \le n} |g_j (x)| \ge \delta'' > 0 \quad \mbox{for all} \; x \in \Delta^d \; .
\]
The $p$-adic Corona theorem then gives $(g_1 , \ldots , g_n) = C \langle X_1 , \ldots , X_d \rangle$. Thus condition 1 follows in this case.

\begin{prop}
  \label{t11}
The preceeding conjecture is true for $d = 1$.
\end{prop}

\begin{proof}
  As explained above, we may assume that all functions $g_1 , \ldots , g_n$ vanish at $x = 0$. Then $f_j (X) = X^{-1} g_j (X)$ is in $C \langle X \rangle$ for every $1 \le j \le n$ and estimate \eqref{eq:14} implies the estimate
\[
\max_{1 \le j \le n} |f_j (x)| \ge \delta \quad \mbox{for all} \; x \in \Delta^1 \; .
\]
The $p$-adic Corona theorem for $d = 1$ now shows that
\[
(f_1 , \ldots , f_n) = (1) \quad \mbox{and hence} \quad (g_1 , \ldots , g_n) = (X) \; .
\]
\end{proof}

Let us now return to $p$-divisible groups and recall the surjection \eqref{eq:11}:
\begin{equation}
  \label{eq:15}
  I / J \twoheadrightarrow \Coker (\eo \to A \hat{\otimes}_R \eo / \tJ) \; .
\end{equation}
Here $I = (X_1 , \ldots , X_d)$ in $\eo [[X_1 , \ldots , X_d]]$ and $J$ is the ideal generated by the elements $h-1$ for $h \in TG'$. Let $J_0 \subset J$ be the ideal generated by the elements $h_1 - 1 , \ldots , h_r -1$ where $h_1 , \ldots , h_r$ form a $\Z_p$-basis of $TG'$. In proposition \ref{t8} we have seen that for some $\delta > 0$ we have
\[
\max_{1 \le j \le r} |h_j (x) - 1| \ge \delta \|x \| \quad \mbox{for all} \; x \in \Delta^d \; .
\]
Conjecture \ref{t10} (which is true for $d = 1$) would therefore imply
\[
(h_1 -1 , \ldots , h_r -1) = (X_1 , \ldots , X_d) \quad\mbox{in} \; C \langle X_1 , \ldots , X_d \rangle \; .
\]
Thus we would find some $t \ge 1$, such that we have
\[
p^t X_i \in J_0 \subset \eo [[ X_1 , \ldots , X_d]] \quad \mbox{for all} \; 1 \le j \le r
\]
and hence also $p^t I \subset J_0 \subset J$. Using the surjection \eqref{eq:15} this would prove claim \ref{t6} and hence theorem \ref{t3} without restriction on $\dim G$. Also theorem \ref{t1} would follow without restriction on $\dim G'$. As it is we have to assume $\dim G \le 1$ resp. $\dim G' \le 1$ in these assertions.

\end{document}